\documentclass[11pt,a4paper]{amsart}

\usepackage{amsmath,amsthm,amssymb}
\usepackage{tikz}
\usepackage{caption}

\usetikzlibrary{arrows,decorations.pathmorphing,decorations.pathreplacing}

\tikzset{vertex/.style={circle,fill=black,inner sep=1pt,outer sep=2pt},
         mvertex/.style={rectangle,draw=black,thick,inner sep=2pt,outer sep=2pt},
         tvertex/.style={inner sep=1pt,font=\scriptsize},
         unvertex/.style={circle,fill=white,draw=white,inner sep=1pt},
         fill1/.style={fill=black!20,draw=black!20},
         fill2/.style={fill=black!40,draw=black!40},
         fill12/.style={fill=black!60,draw=black!60},
         >=stealth',
         leadsto/.style={-angle 90,decorate,decoration=snake,very thick},
         cut/.style={decorate,decoration=saw,very thick}}

\newtheorem{theorem}{Theorem}[section]

\newtheorem{corollary}[theorem]{Corollary}

\newtheorem{lemma}[theorem]{Lemma}
\newtheorem{proposition}[theorem]{Proposition}
\theoremstyle{definition}
\newtheorem{definition}[theorem]{Definition}
\theoremstyle{definition}
\newtheorem{example}[theorem]{Example}
\theoremstyle{remark}
\newtheorem*{remark}{Remark}

\newcommand{\add}{{\operatorname{add}\nolimits}}

\newcommand{\End}{\operatorname{End}\nolimits}
\newcommand{\Ext}{\operatorname{Ext}\nolimits}
\newcommand{\Hom}{\operatorname{Hom}\nolimits}

\renewcommand{\mod}{\operatorname{mod}\nolimits}

\newcommand{\op}{{\operatorname{op}\nolimits}}

\newcommand{\replacevertex}[3][fill=white,draw=white]
 {
  \node at #2 [#1,circle,inner sep=1pt] {};
  \node #2 at #2 #3;
 }
\newcommand{\etalchar}[1]{$^{#1}$}

\newcommand{\cA}{ \ensuremath{ {\mathcal A} } }
\newcommand{\cB}{ \ensuremath{ {\mathcal B} } }
\newcommand{\cC}{ \ensuremath{ {\mathcal C} } }
\newcommand{\cD}{ \ensuremath{ {\mathcal D} } }

\numberwithin{figure}{section}

\title[Finding a CTO for a representation finite CTA]{Finding a cluster-tilting object for a representation finite cluster-tilted algebra}
\date{\today}

\author[Bertani-{\O}kland]{Marco Angel Bertani-{\O}kland}
\author[Oppermann]{Steffen Oppermann}
\author[Wr{\aa}lsen]{Anette Wr{\aa}lsen}

\address{Institutt for matematiske fag\\ NTNU\\ 7491 Trondheim\\ Norway}
\email{Marco.Tepetla@math.ntnu.no}
\email{Steffen.Oppermann@math.ntnu.no}
\email{Anette.Wralsen@math.ntnu.no}

\begin{document}

\begin{abstract}
We provide a technique to find a cluster-tilting object having a given cluster-tilted algebra as endomorphism ring in the finite type case.
\end{abstract}

\maketitle

\section{Introduction}

The cluster category of a finite dimensional hereditary algebra $H$ over an algebraically closed field $k$ was first introduced in \cite{BMRRT}. It provides a categorical model for the cluster algebras of Fomin and Zelevinsky \cite{FZ}, as well as a generalization of the classical tilting theory. The first link from cluster algebras to tilting theory was given in \cite{MRZ} with \emph{decorated representations of quivers}, and later the cluster category was defined as a factor category of the derived category of $H$. 

In the cluster category we have a special class of objects, namely the \emph{cluster-tilting objects}. They induce the \emph{cluster-tilted algebras}, which have been extensively studied (see for instance \cite{BMR1}, \cite{BMR2}, \cite{BMR3}). The topic of this paper is the relationship between cluster-tilting objects and cluster-tilted algebras in the finite type case. More specifically we will study the distribution of cluster-tilting objects in the Auslander Reiten quiver (AR-quiver for short) of the cluster category and the quivers of cluster-tilted algebras.

Given a cluster-tilting object in a cluster category $\mathcal C$ of finite type and its distribution within the AR-quiver of $\cC$, it is easy to read off the quiver with relations of the cluster-tilted algebra induced by it. We present a technique to solve the opposite problem - namely, given the quiver with relations of a cluster-tilted algebra of finite type, how can we find a cluster-tilting object having it as its endomorphism ring? The technique uses a correspondence between subfactor categories of the cluster category and subquivers of $Q$. Along the way we obtain a new proof of the classification of quivers of cluster tilted algebras of type $D_n$ (first obtained in \cite{Dagfinn}).

\section{Background} \label{section_preliminaries}

\subsection{Cluster categories and cluster-tilting objects}

Let $H$ be a finite dimensional hereditary algebra $H$ over an algebraically closed field $k$. We then define its cluster category as an orbit category of its bounded derived category, namely the category $\mathcal C_H = D^b(H)/F$, where $F$ is the composition of the automorphism $\tau^{-1}$, the inverse of the AR-translate of $D^b(H)$, with $[1]$, the shift functor. The objects of $\cC$ are orbits of the objects of $\cD$, while $\Hom_\cC(A,B)=\oplus_i \Hom_\cD(A,F^i B)$ (see \cite{BMRRT} for more details). Since the indecomposable objects of $D^b(H)$ are equivalent to stalk complexes, we can consider them shifts of objects in $\mod H$, and thus all morphisms in $\mathcal C_H$ are induced by morphisms and extensions in $\mod H$.

A \emph{cluster-tilting object} in $\mathcal C_H$ is an object $T$ such that $\Ext^1_\cC(T,X)=0$ if and only if $X$ is in $\add T$ for any object $X \in \cC$. The cluster-tilting objects coincide with the \emph{maximal rigid objects}, that is, objects that are rigid (i.e. with no self-extensions) and are maximal with this property (\cite{BMRRT}). Throughout this paper we will assume that $T$ is basic. These objects can be considered to be generalizations of tilting modules, and all such objects are in fact induced by tilting modules in $\mod H'$ for some hereditary algebra $H'$ derived equivalent to $H$ (\cite[3.3]{BMRRT}). In particular they have exactly $n$ nonisomorphic indecomposable summands, where $n$ is the number of nonisomorphic simple modules of $H$.

In this paper we will only be concerned with cluster categories of finite type. This means that $H$ is Morita equivalent to the path algebra $kQ$ where $k$ is an algebraically closed field and $Q$ is a simply laced Dynkin quiver. In these cases the combinatorics of the AR-quivers of $\mod H$ is very well-known, and we will rely heavily on this.

\subsection{Cluster-tilted algebras of finite type}

A \emph{cluster-tilted algebra} is the endomorphism algebra $B = \End_\cC(T)^{\op}$ of a cluster-tilting object $T$ in $\cC$. By \cite{BMR2} we know that $B$ is of finite representation type if and only if $H$ is of finite representation type. If $H$ is of Dynkin type $\Delta$, we say that $B$ is cluster-tilted of type $\Delta$. 

Finite type cluster-tilted algebras are (up to Morita equivalence) determined uniquely by their quiver by \cite{BMR1}. Furthermore, under our assumptions they are of the form $kQ/I$, where $Q$ is a finite quiver and $I$ some finitely generated admissible ideal in the path algebra $kQ$. 

If the elements of $I$ are linear combinations $k_1 \rho_1 + \cdots + k_n \rho_n$ of paths $\rho_i$ in $Q$, all starting and ending at the same vertex, such that each $k_i$ is non-zero in $k$, they are called \emph{relations}. A \emph{zero-relation} is a relation such that $n=1$. If $n=2$, we call it a \emph{commutativity-relation}. A relation $r$ is called \emph{minimal} if we have that whenever $r=\sum \alpha_i \rho_i \beta_i$, where $\rho_i$ is a relation for every $i$, then there exists an index $j$ such that both $\alpha_j$ and $\beta_j$ are scalars. Furthermore, whenever there is an arrow $b \to a$, a path from $a$ to $b$ is called \emph{shortest} if it contains no proper subpath which is a cycle and if the full subquiver generated by the induced oriented cycle contains no further arrows. We also say that two paths from $j$ to $i$ are {\em disjoint} if they have no common vertices except $j$ and $i$, and we say that two disjoint paths from $j$ to $i$ are {\em disconnected} if the full subquiver generated by the paths contains no further arrows except possibly an arrow from $i$ to $j$. More details on this can be found in \cite{BMR1}. Then the following theorem sums up many properties of $I$:

\begin{theorem}\cite{BMR1}  \label{theorem.relations}
   Let Q be a finite quiver, $k$ an algebraically closed field and $I$ an ideal in the path algebra $kQ$, such that $B=kQ/I$ is a cluster-tilted algebra of finite representation type, and let $a$ and $b$ be vertices of $Q$.
   \begin{itemize}
    \item Assume that there is an arrow $b\to a$. Then there are at most two shortest paths from $a$ to $b$.
	   \begin{itemize}    		
	    \item If there is exactly one, this is a minimal zero-relation.
	    \item If there are two, $r$ and $s$, then $r$ and $s$ are not zero in $B$, they are disconnected, and there is a minimal commutativity-relation $r + \lambda s$ for some $\lambda\neq 0$ in $k$.
	   \end{itemize}
     \item Up to multiplication by non-zero elements of $k$, there are no other zero-relations or commutativity-relations.
     \item The ideal $I$ is generated by minimal zero-relations and minimal commutativity-relations.		
   \end{itemize}
  \end{theorem}

\subsection{Realizing smaller cluster categories inside bigger ones} \label{lokalisering}

Throughout this section we fix a hereditary algebra $H$ and an indecomposable rigid object $M$ in $\cD$.

Several people have noted a class of subfactor categories of the cluster category that are equivalent to smaller cluster categories (smaller here meaning that they are induced by a hereditary algebra with fewer isoclasses of simples). See for instance \cite{BMR3}, \cite{CK} or \cite{IY}. The latter paper treats this for a more general setup, namely $2$-CY triangulated categories with a cluster-tilting object. We thus have the following proposition:

\begin{proposition} \label{proposition.aperp}
Let $\mathcal{C}$ be the cluster category of a hereditary algebra $H$, and let $M \in \mathcal{C}$ be indecomposable and rigid. Then
\[ M^{\perp}_{\mathcal{C}} = \{X \in \mathcal{C} \mid \Hom_{\mathcal{C}}(X, M[1]) = 0 \} / (M), \]
(the category obtained from the subcategory of $\cC$ consisting of all objects not having extensions with $M$ by killing maps factoring through $M$) is equivalent to the cluster category $\mathcal{C}_{H'}$ of a hereditary algebra $H'$ with exactly one fewer isoclasses of simple objects than $H$.

If $M$ is a shift of an indecomposable projective module $He$, then $M^{\perp}_{\mathcal{C}}$ is the cluster category of $H/HeH$.
\end{proposition}

This concrete realization of $\mathcal C_{H'}$ in $\mathcal C_H$ is fundamental for this paper. We give an example below:

\begin{example}
Let $H$ be of Dynkin type $A_6$, and let $M$ be the object of $\cC_H$ indicated in the following drawing of the AR-quiver of $\cC$ (the dashed lines indicate which objects are identified under $F$): 

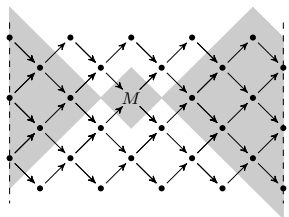
\[ \scalebox{.8}{ \begin{tikzpicture}[scale=.5,yscale=-1]
 % gray areas
 \fill [fill1] (0,0) -- (.5,.5) -- (3,3) -- (4,2) -- (5,3) -- (8,0) -- (9,1) -- (9,7) -- (5,3) -- (4,4) -- (3,3) -- (0,6) -- cycle;
 % distribute vertices
 \foreach \x in {0,...,4}
  \foreach \y in {1,3,5}
   \node (\y-\x) at (\x*2,\y) [vertex] {};
 \foreach \x in {0,...,4}
  \foreach \y in {2,4,6}
   \node (\y-\x) at (\x*2+1,\y) [vertex] {};
 %name M
 \replacevertex[fill1]{(3-2)}{[tvertex] {$M$}}
 %draw arrows
 \foreach \xa/\xb in {0/1,1/2,2/3,3/4}
  \foreach \ya/\yb in {1/2,3/2,3/4,5/4,5/6}
   {
    \draw [->] (\ya-\xa) -- (\yb-\xa);
    \draw [->] (\yb-\xa) -- (\ya-\xb);
   }
 \foreach \x in {0,...,4}
  \foreach \ya/\yb in {1/2,3/2,3/4,5/4,5/6}
   {
    \draw [->] (\ya-\x) -- (\yb-\x);
   }  
 %draw dashed lines
 \draw [dashed] (0,.5) -- (0,6.5); 
 \draw [dashed] (9,.5) -- (9,6.5);
\end{tikzpicture} } \]

The objects that are shaded here are the objects \emph{not} in $M^{\perp}_{\mathcal{C}}$. It is not hard to see that the remaining objects (modulo all maps factoring through $M$) form a cluster category of type $A_2\coprod A_3$.
\end{example}

We then have the following useful corollary:

\begin{corollary}[{\cite[4.9]{IY}}] \label{corollary.sametilt}
Let $M$ be a rigid indecomposable object in the cluster category $\mathcal{C}$ of $H$. Then there is a bijection between cluster-tilting objects in $\mathcal{C}$ containing $M$ and cluster-tilting objects in $\mathcal{C}_{H'}$.
\end{corollary}

In particular this means that given a cluster-tilting object $T$ and removing any summand $T'$ will give an object $T/T'$ that is again a cluster-tilting object in a cluster category (though now possibly disconnected) that is equivalent to a subfactor category of the original cluster category.

\section{Subfactor categories of Dynkin type $A$}

In this section we will focus on the quivers of cluster-tilted algebras of type $A$ and how these always can be assumed to be induced by cluster-tilting objects living in a domain of the cluster category corresponding to $\mod kQ_m$ where $Q_m$ is $A_m$ with linear orientation. We will also show that the same happens for certain subquivers of quivers of cluster-tilted algebras of type $D$.

\subsection{Quivers of type $A$ and triangles in cluster categories}

By \cite{BV} the class of quivers of cluster-tilted algebras of type $A$ can be described as follows (equivalent classifications can be found in \cite{Seven} and implicitly in \cite{otherCCS}).

\begin{itemize}
\item [$\bullet$] all non-trivial cycles are oriented and of length $3$
\item [$\bullet$] a vertex has at most four neighbours
\item [$\bullet$] if a vertex has four neighbours, then two of its adjacent arrows belong to one $3$-cycle, and the other two belong to another $3$-cycle
\item [$\bullet$] if a vertex has exactly three neighbours, then two of its adjacent arrows belong to a $3$-cycle, and the third arrow does not belong to any $3$-cycle
\end{itemize}

In particular there are no non-oriented cycles in such a quiver (and thus no multiple arrows). Actually such a quiver is always a full connected subquiver of the infinite quiver $Q_{\cA}$ drawn below:

\begin{minipage}{\textwidth}\label{Q_A}
\[ \begin{tikzpicture}[xscale=.8,yscale=.4,yscale=-1]
 \node (P1) at (13.6,2) [vertex] {};
 \node (P2) at (11.6,4) [vertex] {};
 \node (P3) at (15.6,4) [vertex] {};
 \node (P4) at (10.6,5.5) [vertex] {};
 \node (P5) at (12.6,5.5) [vertex] {};
 \node (P6) at (14.6,5.5) [vertex] {};
 \node (P7) at (16.6,5.5) [vertex] {};
 \node (P8) at (10.1,6.5) [vertex] {};
 \node (P9) at (11.1,6.5) [vertex] {};
 \node (P10) at (12.1,6.5) [vertex] {};
 \node (P11) at (13.1,6.5) [vertex] {};
 \node (P12) at (14.1,6.5) [vertex] {};
 \node (P13) at (15.1,6.5) [vertex] {};
 \node (P14) at (16.1,6.5) [vertex] {};
 \node (P15) at (17.1,6.5) [vertex] {};
 \draw [->] (P2) -- (P1);
 \draw [->] (P1) -- (P3);
 \draw [->] (P3) -- (P2);
 \draw [->] (P4) -- (P2);
 \draw [->] (P2) -- (P5);
 \draw [->] (P6) -- (P3);
 \draw [->] (P3) -- (P7);
 \draw [->] (P5) -- (P4);
 \draw [->] (P7) -- (P6);
 \draw [->] (P8) -- (P4);
 \draw [->] (P4) -- (P9);
 \draw [->] (P10) -- (P5);
 \draw [->] (P5) -- (P11);
 \draw [->] (P12) -- (P6);
 \draw [->] (P6) -- (P13);
 \draw [->] (P14) -- (P7);
 \draw [->] (P7) -- (P15);
 \draw [->] (P15) -- (P14);
 \draw [->] (P13) -- (P12);
 \draw [->] (P11) -- (P10);
 \draw [->] (P9) -- (P8);
 \draw [dotted] (P8) -- +(-.3,.7);
 \draw [dotted] (P9) -- +(.3,.7);
 \draw [dotted] (P10) -- +(-.3,.7);
 \draw [dotted] (P11) -- +(.3,.7);
 \draw [dotted] (P12) -- +(-.3,.7);
 \draw [dotted] (P13) -- +(.3,.7);
 \draw [dotted] (P14) -- +(-.3,.7);
 \draw [dotted] (P15) -- +(.3,.7);
\end{tikzpicture} \]\captionof{figure}{} \label{figure.an}
\end{minipage}
where the composition of any two arrows on a $3$-cycle is zero. In such a quiver we will call a vertex a \emph{connecting vertex} (this notation has been introduced in \cite{Dagfinn}) if
\begin{itemize}
 \item[$\bullet$]there are at most two arrows adjacent with it, and
 \item[$\bullet$]whenever there are two arrows adjacent with it, the vertex is on a $3$-cycle.
\end{itemize} 
For instance, in the quiver
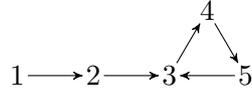
\[ \begin{tikzpicture}[yscale=-1]
 \node (1) at (0,0) [inner sep=1pt] {1};
 \node (2) at (1,0) [inner sep=1pt] {2};
 \node (3) at (2,0) [inner sep=1pt] {3};
 \node (4) at (2.5,-.86) [inner sep=1pt] {4};
 \node (5) at (3,0) [inner sep=1pt] {5};
 \draw [->] (1) -- (2);
 \draw [->] (2) -- (3);
 \draw [->] (3) -- (4);
 \draw [->] (4) -- (5);
 \draw [->] (5) -- (3);
\end{tikzpicture} \]
the vertices 1, 4 and 5 are connecting vertices, while the vertices 2 and 3 are not.

Next we define what we mean by an $\mathcal A$-triangle in a quiver. It is simply the quiver of a cluster-tilted algebra of Dynkin type $A_m$ with exactly one connecting vertex marked. It is easy to see that the quiver of any cluster-tilted algebra of type $A$ has at least one connecting vertex, so we just have to choose one. In particular we can always assume that the marked connecting vertex corresponds to the top vertex in the picture of $Q_{\cA}$ in Figure \ref{Q_A}.

Now we move on to a class of subcategories of cluster categories that we will refer to as \emph{triangles of order $m$ in $\cC$}. As already indicated they correspond to domains in the AR-quiver of $\cC$ corresponding to $\mod kQ_m$. Thus any rigid object in such a triangle is induced by a partial tilting module over $kQ_m$, and if the object is maximal within the triangle it corresponds to a tilting module. Below are two examples of triangles of order $2$ and $4$ in a cluster category of type $D_7$:

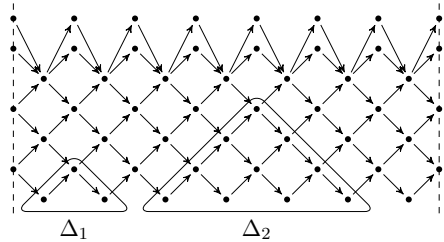
\[ \scalebox{.8}{ \begin{tikzpicture}[scale=.5,yscale=-1]
 \foreach \x in {0,...,7}
  \foreach \y in {1,2,4,6}
   \node (\y-\x) at (\x*2,\y) [vertex] {};
 \foreach \x in {0,...,6}
  \foreach \y in {3,5,7}
   \node (\y-\x) at (\x*2+1,\y) [vertex] {};
 %draw arrows
 \foreach \xa/\xb in {0/1,1/2,2/3,3/4,4/5,5/6,6/7}
  \foreach \ya/\yb in {1/3,2/3,4/3,4/5,6/5,6/7}
   {
    \draw [->] (\ya-\xa) -- (\yb-\xa);
    \draw [->] (\yb-\xa) -- (\ya-\xb);
   }
 %draw dashed lines
 \draw [dashed] (0,.5) -- (0,7.5); 
 \draw [dashed] (14,.5) -- (14,7.5);
 %the categories we are talking about
 \draw [rounded corners=5pt] (2,5.5) -- (3.9,7.4) -- node [below] {$\Delta_1$} (.1,7.4) -- cycle;
 \draw [rounded corners=5pt] (8,3.5) -- (11.9,7.4) -- node [below] {$\Delta_2$} (4.1,7.4) -- cycle;
\end{tikzpicture} } \] 

We will also denote the object in a triangle $\Delta$ that corresponds to the projective-injective module over $kQ_m$ by $\Pi(\Delta)$. Since every maximal rigid object in $\Delta$ corresponds to a tilting module over $kQ_m$, it will always have $\Pi(\Delta)$ as a summand.

Next we will show that any cluster-tilting object of type $A_m$ is induced by a tilting module over $\mod kQ_m$.

\begin{lemma}\label{oppe og nede}
Every cluster-tilting object $T$ in the cluster category $\mathcal C_m$ of type $A_m$ has at least two summands in the outmost $\tau$-orbit in the AR-quiver of $\mathcal C_m$.
\end{lemma}

\begin{proof}
Clearly this is the case for $\mathcal C_2$.

Now assume that the claim holds for any $m'<m$. Since $m$ can be assumed to be at least $3$, either the result holds or $T$ must have at least one indecomposable summand $T_*$ that is not in the outer $\tau$-orbit of the AR-quiver. We then know that $T_*^{\perp}$ looks as follows in $\mathcal C_m$: \\
\begin{minipage}{\textwidth}
\[ \scalebox{.8} { \begin{tikzpicture}[scale=.5,yscale=-1]
 % gray areas
 \fill [fill1] (0,.5) -- (.5,.5) -- (5,5) -- (9.5,.5) -- (10,.5) -- (10,7.5) -- (9.5,7.5) -- (5,3) -- (.5,7.5) -- (0,7.5) -- cycle;
 % distribute vertices
 \foreach \x in {0,...,5}
  \foreach \y in {1,3,5,7}
   \node (\y-\x) at (\x*2,\y) [vertex] {};
 \foreach \x in {0,...,4}
  \foreach \y in {2,4,6}
   \node (\y-\x) at (\x*2+1,\y) [vertex] {};
 \replacevertex[fill1]{(4-2)}{[tvertex] {$T_*$}}
 %draw arrows
 \foreach \xa/\xb in {0/1,1/2,2/3,3/4,4/5}
  \foreach \ya/\yb in {1/2,3/4,5/4,7/6}
   {
    \draw [->] (\ya-\xa) -- (\yb-\xa);
    \draw [->] (\yb-\xa) -- (\ya-\xb);
   }
 \foreach \xa/\xb in {0/1,1/2,2/3,3/4,4/5}
  \foreach \ya/\yb in {3/2,5/6}
   {
    \draw [thick,loosely dotted] (\ya-\xa) -- (\yb-\xa);
    \draw [thick,loosely dotted] (\yb-\xa) -- (\ya-\xb);
   }
 %draw dashed lines
 \draw [dashed] (0,.5) -- (0,7.5); 
 \draw [dashed] (10,.5) -- (10,7.5);
 \draw [decorate,decoration=brace] (1.5,.5) -- node [above] {$C_{m-k}$} (8.5,.5);
 \draw [decorate,decoration={brace,mirror}] (1.5,7.5) -- node [below] {$C_{k-1}$} (8.5,7.5);
\end{tikzpicture} } \]
\captionof{figure}{} \label{figure.oppe_og_nede}
\end{minipage}
$T_*^{\perp}$ is equivalent to $\mathcal C_{k-1} \times \mathcal C_{m-k}$, and we can, by Corollary~\ref{corollary.sametilt}, write $T=T_* \oplus T_1 \oplus T_2$ where $T_1$ is a cluster-tilting object in $\mathcal C_{k-1}$ and $T_2$ is a cluster-tilting object in $\mathcal C_{m-k}$.

The domain of $\mathcal C_{k-1}$ is the objects below $T_*$ in the above figure. By assumption every cluster-tilting object in $\mathcal C_{k-1}$ must have at least two summands in the top $\tau$-orbit of its AR-quiver. We note that this cannot be the two objects directly below $T_*$ since there is a non-zero extension between them. Thus $T_1$ must have at least one summand in the bottom row of the AR-quiver of $\mathcal C_{k-1}$ and thus of $\mathcal C_m$. A similar argument holds for $T_2$ and $\mathcal C_{m-k}$, and so we are done.
\end{proof}

\begin{lemma}\label{alltid linear}
Every cluster-tilting object $T$ in the cluster category $\cC_m$ is induced by a tilting module over $kQ_m$, where $Q_m$ is of type $A_m$ with linear orientation.
\end{lemma}

\begin{proof}
This result follows from the very simple observation that the objects that have non-zero extensions with an object in the outmost row of the AR-quiver can be seen as $kQ_m[1]$ for some embedding of $kQ_m$ into $\mathcal{C}_m$. Since none of these objects can be summands of $T$, the result follows.
\end{proof}

The next proposition tells us which summands of a cluster-tilting object $T$ can induce connecting vertices in the quiver of $\End(T)^{\op}$:

\begin{proposition}
Let $T$ be a cluster-tilting object in the cluster category $\mathcal C_m$ of type $A_m$. Then an indecomposable summand $T_*$ of $T$ induces a connecting vertex in the quiver $Q$ of $\End(T)^{\op}$ if and only if it is in the outmost $\tau$-orbit in the AR-quiver of $\mathcal C_m$.
\end{proposition}

\begin{proof}
First, from the AR-quiver of $\mathcal C_m$ it is clear that an object in the outmost $\tau$-orbit in the AR-quiver will induce a connecting vertex in $Q$. 

Next, assume that an indecomposable summand $T_*$ of $T$ that is not in the outmost row of the AR-quiver of $\mathcal C_m$ induces a connecting vertex. Consider Figure~\ref{figure.oppe_og_nede}. Since $T_*$ is not on the outmost row, $T_*^{\perp}$ decomposes into two non-connected subcategories containing $k-1$ and $n-k$ summands of $T$ respectively.

If there is only one other indecomposable summand of $T$ with an irreducible map to or from $T_*$, this would mean that $Q$ is a disconnected quiver since there must be at least one summand of $T$ in each subcategory. Thus this cannot happen. 

Hence assume that $T_*$ corresponds to a vertex on a $3$-cycle. In this case the other summands in the cycle would have to be either both in $\mathcal C_{k-1}$ or both in $\mathcal C_{m-k}$, since there is no non-zero map completing the triangle otherwise. Again this would disconnect the quiver, and so it cannot happen.
\end{proof}

The following corollary is an immediate consequence of this proposition.

\begin{corollary}\label{triangelkorrespondanse}
Every $\mathcal A$-triangle and every maximal rigid object in a triangle of order $m$ in the cluster category $\mathcal{C}$ of type $D_n$ is induced by a tilting module over $\mod kQ_m$ where $Q_m$ is of type $A_m$ with linear orientation. Thus every $\mathcal A$-triangle corresponds to some maximal rigid object in a triangle in $\mathcal C$, and every maximal rigid object in a triangle in $\mathcal C$ induces an $\mathcal A$-triangle.
\end{corollary}

\subsection{Subcategories of Dynkin type $A$ in the cluster category of type $D$}

In this subsection we will develop some lemmas which will be used to show the main results of this paper. Throughout this subsection we will assume that $\cC$ is a cluster category of Dynkin type $D$.

From now on let $H = kQ$, where $Q$ is some orientation of $D_n$. Then the AR-quiver of $\mathcal D$ is of the form $\mathbb ZQ$ and can be drawn as follows:
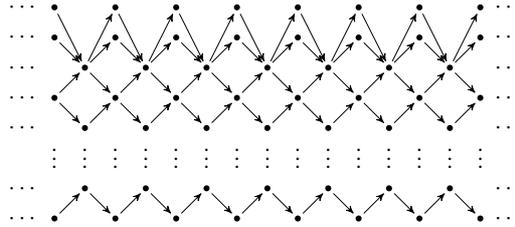
\[ \scalebox{.8}{ \begin{tikzpicture}[scale=.5,yscale=-1]
 % vertices
 \foreach \x in {0,...,7}
  \foreach \y in {1,2,4,8}
   \node (\y-\x) at (\x*2,\y) [vertex] {};
 \foreach \x in {0,...,6}
  \foreach \y in {3,5,7}
   \node (\y-\x) at (\x*2+1,\y) [vertex] {};
 % arrows
 \foreach \xa/\xb in {0/1,1/2,2/3,3/4,4/5,5/6,6/7}
  \foreach \ya/\yb in {1/3,2/3,4/3,4/5,8/7}
   {
    \draw [->] (\ya-\xa) -- (\yb-\xa);
    \draw [->] (\yb-\xa) -- (\ya-\xb);
   }
 % dots
 \foreach \x in {0,...,14}
  \node at (\x,5.8) {$\vdots$};
 \foreach \y in {1,...,5,7,8}
  {
   \node at (-1,\y) {$\cdots$};
   \node at (15,\y) {$\cdots$};
  }
\end{tikzpicture} } \]
When the AR-quiver of $\mathcal C_H$ is drawn as above, we will refer to the objects in the top two rows of the AR-quiver of $\mathcal D$ as \emph{$\alpha$-objects} and to the remaining objects as \emph{$\beta$-objects}. 

\begin{definition}
Let $A$ be an $\alpha$-object. Then $\phi A$ is the unique other $\alpha$-object which occurs together with $A$ as a summand of the middle term of an almost split triangle.
\end{definition}

This means that $\phi A$ is the $\alpha$-object drawn directly below or above $A$ in the AR-quiver above.
The Ext-supports of the objects in this category are well-known. Assuming that $A$ is an $\alpha$-object in $\mathcal C_H$, we will  illustrate $A^{\perp}_{\mathcal C}$ in the AR-quiver of $\mathcal C_H$ as follows.
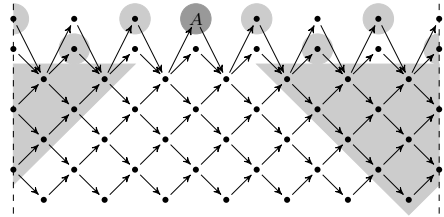
\[ \scalebox{.8}{ \begin{tikzpicture}[scale=.5,yscale=-1]
 %\grey areas
 \foreach \x in {2,4,6}
  \fill [fill1] (\x*2,1) circle (.5);
 \fill [fill2] (6,1) circle (.5);
 \foreach \x in {1,5}
  \fill [fill1] (\x*2,2) circle (.5);
 \fill [fill1] (0,.5) arc (-90:90:.5);
 \fill [fill1] (14,2.5) arc (90:270:.5);
 \fill [fill1] (0,2.5) -- (4,2.5) -- (0,6.5) -- cycle;
 \fill [fill1] (8,2.5) -- (14,2.5) -- (14,6.5) -- (13,7.5) -- cycle;
 % distribute vertices
 \foreach \x in {0,...,7}
  \foreach \y in {1,2,4,6}
   \node (\y-\x) at (\x*2,\y) [vertex] {};
 \foreach \x in {0,...,6}
  \foreach \y in {3,5,7}
   \node (\y-\x) at (\x*2+1,\y) [vertex] {};
 \replacevertex[fill2]{(1-3)}{[tvertex] {$A$}}
 %draw arrows
 \foreach \xa/\xb in {0/1,1/2,2/3,3/4,4/5,5/6,6/7}
  \foreach \ya/\yb in {1/3,2/3,4/3,4/5,6/5,6/7}
   {
    \draw [->] (\ya-\xa) -- (\yb-\xa);
    \draw [->] (\yb-\xa) -- (\ya-\xb);
   }
 %draw dashed lines
 \draw [dashed] (0,.5) -- (0,7.5); 
 \draw [dashed] (14,.5) -- (14,7.5);
\end{tikzpicture} } \]
Here the dashed lines indicate which objects are identified by the functor $F$.

Assume that $T$ is a cluster-tilting object having $A$ as a summand. By Corollary~\ref{corollary.sametilt} we have the following lemma:

\begin{lemma}\label{An-delen_bevares}
Let $T$ be a cluster-tilting object in $C_H$, and assume that $A$ is a summand in $T$. Then $T = T_1 \oplus T_2$ where $T_1$ is the sum of all indecomposable $\alpha$-objects in $T$ and $T_2$ is the sum of all $\beta$-objects. If $T'$ is the image of $T$ in $A^{\perp}_{\mathcal C}$, then $T' = T_1' \oplus T_2'$ such that $T_1'$ is isomorphic to $T_1/A$ and $T_2'$ is isomorphic to $T_2$.
\end{lemma}

The category $\cC' = A^\perp_\cC$ is a cluster category of Dynkin type $A$ by Theorem~\ref{proposition.aperp}. We set $\mathcal B(A) = (\phi A)^{\perp}_{\mathcal C'}$. This category is the cluster category of some product of path algebras of Dynkin type $A$. In particular it must be the intersection of $A^{\perp}_{\mathcal C}$ and $(\phi A)^{\perp}_{\mathcal C}$, and thus it consists of all $\beta$-objects in $A^{\perp}_{\mathcal C}$. From the above figure we see that this is the domain of the cluster category of Dynkin type $A_{n-2}$. 

As demonstrated in the previous section, there are many triangles in $\mathcal C$. We now define two particular triangles $\mathcal B_*(A)$ and $_*{\mathcal B}(A)$ and the subcategory $\cB(A)$.

\begin{definition}
Let $\mathcal C$, $A$ and $A^{\perp}_{\mathcal C}$ be as above. Then we define $\mathcal B_*(A)$ to be the subcategory $\{B \in \mathcal B(A) | \Hom_{\mathcal C}(B,A)=0\}$ and $_*{\mathcal B}(A)$ to be the subcategory $\{B  \in \mathcal B(A) | \Hom_{\mathcal C}(A,B)=0\}$ in $\mathcal C$. We also define $\cB(A) = {\mathcal B}_*(A) \cup _*{\mathcal B}(A)$.
\end{definition}

\noindent
\begin{minipage}{\textwidth}
\[ \scalebox{.8}{ \begin{tikzpicture}[scale=.5,yscale=-1,baseline=-40pt]
 %\grey areas
 \foreach \x in {2,4,6}
  \fill [fill1] (\x*2,1) circle (.5);
 \foreach \x in {1,5}
  \fill [fill1] (\x*2,2) circle (.5);
 \fill [fill1] (0,.5) arc (-90:90:.5);
 \fill [fill1] (14,2.5) arc (90:270:.5);
 \fill [fill1] (0,2.5) -- (4,2.5) -- (0,6.5) -- cycle;
 \fill [fill1] (8,2.5) -- (14,2.5) -- (14,6.5) -- (13,7.5) -- cycle;
 % distribute vertices
 \foreach \x in {0,...,7}
  \foreach \y in {1,2,4,6}
   \node (\y-\x) at (\x*2,\y) [vertex] {};
 \foreach \x in {0,...,6}
  \foreach \y in {3,5,7}
   \node (\y-\x) at (\x*2+1,\y) [vertex] {};
 \replacevertex{(1-3)}{[tvertex] {$A$}}
 %draw arrows
 \foreach \xa/\xb in {0/1,1/2,2/3,3/4,4/5,5/6,6/7}
  \foreach \ya/\yb in {1/3,2/3,4/3,4/5,6/5,6/7}
   {
    \draw [->] (\ya-\xa) -- (\yb-\xa);
    \draw [->] (\yb-\xa) -- (\ya-\xb);
   }
 %draw dashed lines
 \draw [dashed] (0,.5) -- (0,7.5); 
 \draw [dashed] (14,.5) -- (14,7.5);
 %the categories we are talking about
 \draw [rounded corners=5pt] (5,2.5) -- (9.8,7.3) -- node [below,pos=.92] {$_*{\mathcal B}$} (.2,7.3) -- cycle;
 \draw [rounded corners=5pt] (7,2.5) -- (12.2,7.7) -- node [below,very near start] {$\mathcal{B}_*$} (1.8,7.7) -- cycle;
\end{tikzpicture} } \]
\captionof{figure}{} \label{figure.Bstar}
\end{minipage}

We then have the following lemma which is easily seen to be true by Figure \ref{figure.Bstar}.

\begin{lemma}\label{triangler1}
If $\Delta$ is a triangle in $\mathcal B(A)$, then any map from $A$ to an object in $\Delta$ and any map from an object in $\Delta$ to $A$ must factor through $\Pi(\Delta)$.
\end{lemma}

Next we give a lemma that explains some of the relationship between $\cB(A_1)$ and $\cB(A_2)$ for two $\alpha$-objects $A_1$ and $A_2$ in $\cC$.

\begin{lemma} \label{triangler2}
Let $A_1$ and $A_2$ be two $\alpha$-objects such that $A_1 = \tau^kA_2$ or $\tau^k(\phi A_2)$ for some $k$, $0<k<n$. Then $\mathcal B(A_1)\cap\mathcal B(A_2)$ is the union of two disjoint triangles $\Delta_1$ and $\Delta_2$ of order $k-1$ and $n-k-1$, respectively, such that
\begin{itemize}
\item[1)] every map from $A_1$ to $A_2$ factors through $\Pi(\Delta_2)$,
\item[2)] every map $f: X_1 \rightarrow X_2$ where $X_1 \in \Delta_1$ and $X_2 \in \Delta_2$ factors through $A_1$, 
\item[3)] there is a non-zero map $f: \Pi(\Delta_2) \rightarrow \Pi(\Delta_1)$ which factors through $A_2$ and
\item[4)] $\Hom_{\mathcal C_H}(X,A_1)=0$ for $X \in \Delta_2$ and $\Hom_{\mathcal C_H}(A_1,Y)=0$ for any $Y \in \Delta_1$.
\end{itemize}
\end{lemma}

\begin{proof}
By considering $\mathcal B(A) \cap \mathcal B(\tau^k A)$ for $0<k<n$ we see that $\mathcal B(A_1) \cap \mathcal B(A_2)$ is the union of two disjoint triangles of order $k-1$ and $n-k-1$. The following diagram illustrates this:

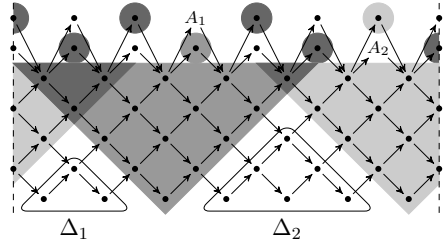
\[ \scalebox{.8}{ \begin{tikzpicture}[scale=.5,yscale=-1]
  %\grey areas
 \foreach \x in {2,4}
  \fill [fill12] (\x*2,1) circle (.5);
 \foreach \x in {1,5}
  \fill [fill12] (\x*2,2) circle (.5);
 \fill [fill2] (6,2) circle (.5);
 \fill [fill1] (12,1) circle (.5);
 \fill [fill12] (0,.5) arc (-90:90:.5);
 \fill [fill12] (14,2.5) arc (90:270:.5);
 \fill [fill1] (0,2.5) -- (4,2.5) -- (0,6.5) -- cycle;
 \fill [fill1] (8,2.5) -- (14,2.5) -- (14,6.5) -- (13,7.5) -- cycle;
 \fill [fill2] (0,2.5) -- (5,7.5) -- (10,2.5) -- cycle;
 \fill [fill12] (0,2.5) -- (2,4.5) -- (4,2.5) -- cycle;
 \fill [fill12] (8,2.5) -- (9,3.5) -- (10,2.5) -- cycle;
 % distribute vertices
 \foreach \x in {0,...,7}
  \foreach \y in {1,2,4,6}
   \node (\y-\x) at (\x*2,\y) [vertex] {};
 \foreach \x in {0,...,6}
  \foreach \y in {3,5,7}
   \node (\y-\x) at (\x*2+1,\y) [vertex] {};
 \replacevertex{(1-3)}{[tvertex] {$A_1$}}
 \replacevertex{(2-6)}{[tvertex] {$A_2$}}
 %draw arrows
 \foreach \xa/\xb in {0/1,1/2,2/3,3/4,4/5,5/6,6/7}
  \foreach \ya/\yb in {1/3,2/3,4/3,4/5,6/5,6/7}
   {
    \draw [->] (\ya-\xa) -- (\yb-\xa);
    \draw [->] (\yb-\xa) -- (\ya-\xb);
   }
 %draw dashed lines
 \draw [dashed] (0,.5) -- (0,7.5);
 \draw [dashed] (14,.5) -- (14,7.5);
  %triangles
 \draw [rounded corners=5pt] (2,5.5) -- (3.9,7.4) -- node [below] {$\Delta_1$} (.1,7.4) -- cycle;
 \draw [rounded corners=5pt] (9,4.5) -- (11.9,7.4) -- node [below] {$\Delta_2$} (6.1,7.4) -- cycle;
\end{tikzpicture} } \]
It is also clear from this that every map from $A_1$ to $A_2$ factors through $\Pi(\Delta_2)$. Furthermore, if there was a non-zero map $f$ from some object in $\Delta_2$ to $A_1$, it would have to factor through $\Pi(\Delta_2)$ by Lemma~\ref{triangler1}. Since $\Hom_{\mathcal C}(A_1, \Pi(\Delta_2))$ is non-zero, these maps would induce a $2$-cycle in the quiver of any cluster-tilted algebra induced by a cluster-tilting object having $A_1$ and $\Pi(\Delta_2)$ as summands. This is a contradiction to a result of Todorov (see \cite[Proposition~3.2]{BMR3}).

Next we argue that every map $f: X_1 \rightarrow X_2$ where $X_1 \in \Delta_1$ and $X_2 \in \Delta_2$ factors through $A_1$. Since $\mathcal B(A_1)$ corresponds to the cluster category $\mathcal C''$ of Dynkin type $A_{n-2}$, it is clear from the figure that $\Hom_{\mathcal C''}(X_1,X_2)=0$ for $X_1 \in \Delta_1$, $X_2 \in \Delta_2$. Thus any non-zero map in $\Hom_{\mathcal C}(X_1,X_2)$ must factor through $A_1$.

Since there is no non-zero map from an object in $\Delta_1$ to an object in $\Delta_2$ in $\mathcal B(A_1)$, there must be a non-zero map $f: X_2' \rightarrow X_1'$ where $X_2' \in \Delta_2$ and $X_1' \in \Delta_1$ (otherwise $A_1^{\perp}$ would be disconnected). But by the previous paragraph this map factors through $A_2$, and by Lemma~\ref{triangler1} it must also factor through $\Pi(\Delta_2)$ and $\Pi(\Delta_1)$.
\end{proof}

\section{The distribution of cluster-tilting objects of type $D_n$}

In this section we will show how to, given the quiver $Q$ of a cluster-tilted algebra of type $D_n$, explicitly find a cluster-tilting object in the cluster category of type $D_n$ inducing it. The goal is to prove the following theorem:

\begin{theorem} \label{main_theorem}
Given the quiver $Q$ of a cluster-tilted algebra of type $D$, it will be of one of the types

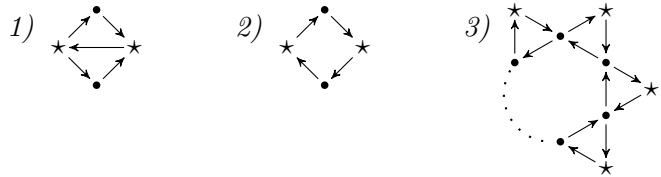
\[ \begin{tikzpicture}[scale=.5,yscale=-1]
 %%first diagram
 %\node at (-1,.5) {1)};
 %\node (P1) at (0,0) [vertex] {};
 %\node (P2) at (2,0) [vertex] {};
 %\node (P3) at (1,1) [inner sep=1pt] {$\star$};
 %\draw (P1) -- (P3);
 %\draw (P2) -- (P3);
 %\node at (0,-1) {};
 % second diagram
 \pgftransformshift{\pgfpoint{-2cm}{0}};
 \node at (1,.5) {1)};
 \node (P1) at (2,1) [inner sep=1pt] {$\star$};
 \node (P2) at (3,0) [vertex] {};
 \node (P3) at (3,2) [vertex] {};
 \node (P4) at (4,1) [inner sep=1pt] {$\star$};
 \draw [->] (P1) -- (P2);
 \draw [->] (P1) -- (P3);
 \draw [->] (P2) -- (P4);
 \draw [->] (P3) -- (P4);
 \draw [->] (P4) -- (P1);
 % third diagram
 \pgftransformshift{\pgfpoint{6cm}{-5cm}};
 \node at (1,5.5) {2)};
 \node (P1) at (2,6) [inner sep=1pt] {$\star$};
 \node (P2) at (3,5) [vertex] {};
 \node (P3) at (3,7) [vertex] {};
 \node (P4) at (4,6) [inner sep=1pt] {$\star$};
 \draw [->] (P1) -- (P2);
 \draw [->] (P3) -- (P1);
 \draw [->] (P2) -- (P4);
 \draw [->] (P4) -- (P3);
 % fourth diagram
 \pgftransformshift{\pgfpoint{6cm}{-3cm}};
 \node at (1,8.5) {3)};
 \node (P1) at (2,9.4) [vertex] {};
 \node (P2) at (3.2,8.7) [vertex] {};
 \node (P3) at (4.4,9.4) [vertex] {};
 \node (P4) at (4.4,10.8) [vertex] {};
 \node (P5) at (3.2,11.5) [vertex] {};
 \node (P6) at (2,8) [inner sep=1pt] {$\star$};
 \node (P7) at (4.4,8) [inner sep=1pt] {$\star$};
 \node (P8) at (5.6,10.1) [inner sep=1pt] {$\star$};
 \node (P9) at (4.4,12.2) [inner sep=1pt] {$\star$};
 \draw [->] (P2) -- (P1);
 \draw [->] (P3) -- (P2);
 \draw [->] (P4) -- (P3);
 \draw [->] (P5) -- (P4);
 \draw [->] (P1) -- (P6);
 \draw [->] (P6) -- (P2);
 \draw [->] (P2) -- (P7);
 \draw [->] (P7) -- (P3);
 \draw [->] (P3) -- (P8);
 \draw [->] (P8) -- (P4);
 \draw [->] (P4) -- (P9);
 \draw [->] (P9) -- (P5);
 \draw [thick, loosely dotted] (P1) .. controls (1.4,10.45) and (2,11.5) .. (P5);
\end{tikzpicture} \]
where the $\star$-vertices are connecting vertices of possibly empty $\cA$-triangles and all relations are according to Theorem~\ref{theorem.relations}. In particular we have the following: 
\begin{itemize}
\item[a)] If $Q$ is of type 1), it is induced by a cluster-tilting object having exactly two $\alpha$-objects $A$ and $\phi A$ as summands.
\item[b)] If $Q$ is of type 2), it is induced by a cluster-tilting object having exactly two $\alpha$-objects $A_1$ and $A_2$ such that $A_2 \neq \phi A_1$ as summands. Furthermore, if the number of vertices in the two $\cA$-triangles are $n_1$ and $n_2$, then we can choose $A_1$ and $A_2$ such that  $A_1$ is in $\{\tau^{n_1+1}A_2, \tau^{n_1+1}\phi A_2\}$ and $A_2$ is in $\{\tau^{n_2+1}A_1, \tau^{n_2+1}\phi A_1\}$.
\item[c)] If $Q$ is of type 3) it is induced by a cluster-tilting object having more than two $\alpha$-objects as summands. Furthermore the $\alpha$-objects will be distributed within the AR-quiver of $\cC$ according to the number of vertices in the $\cA$-triangles as in b).
\end{itemize}
\end{theorem}

\begin{remark}
Note that this is enough to find the distribution of some tilting object $T$ inducing the quiver of a given cluster-tilted algebra, since the number and distribution of the $\alpha$-objects determines the shape of the quiver up to $\cA$-triangles as indicated in the theorem.
\end{remark}

What we will do is to give the correspondence between triangles of order $m$ in $\cC$ and the $\mathcal A$-triangles appearing in the quiver $Q$. We will argue that every $\mathcal A$-triangle is induced by a rigid object in a triangle in the cluster category of type $D_n$ which is maximal within the triangle, and that any such object induces an $\mathcal A$-triangle. First we show that any cluster-tilting object has summands that are $\alpha$-objects:

\begin{lemma} \label{lemma.alpha-objects}
Any cluster-tilting object in the cluster category $\mathcal C$ of type $D_n$ has at least two non-isomorphic $\alpha$-objects as indecomposable summands.
\end{lemma}

\begin{proof}
Assume that $T$ has at least one indecomposable $\beta$-object $T_*$ as a summand (otherwise we are done). Furthermore assume that $T_*$ is in row $k$ of the AR-quiver of $\cC$ (counted from the bottom), and consider the following figure of $T_*^{\perp}$:
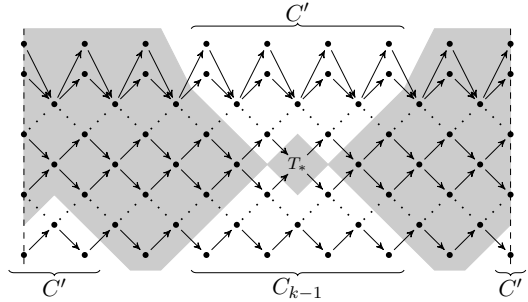
\[ \scalebox{.8} { \begin{tikzpicture}[scale=.5,yscale=-1]
 % gray areas
 \fill [fill1] (0,-.5) -- (4.5,-.5) -- (5.5,1.5) -- (9,5) -- (12.5,1.5) -- (13.5,-.5) -- (16,-.5) -- (16,6) -- (14.5,7.5) -- (13.5,7.5) -- (9,3) -- (4.5,7.5) -- (3.5,7.5) -- (1,5) -- (0,6) -- cycle;
 % distribute vertices
 \foreach \x in {0,...,8}
  \foreach \y in {0,1,3,5,7}
   \node (\y-\x) at (\x*2,\y) [vertex] {};
 \foreach \x in {0,...,7}
  \foreach \y in {2,4,6}
   \node (\y-\x) at (\x*2+1,\y) [vertex] {};
 \replacevertex[fill1]{(4-4)}{[tvertex] {$T_*$}}
 %draw arrows
 \foreach \xa/\xb in {0/1,1/2,2/3,3/4,4/5,5/6,6/7,7/8}
  \foreach \ya/\yb in {0/2,1/2,3/4,5/4,7/6}
   {
    \draw [->] (\ya-\xa) -- (\yb-\xa);
    \draw [->] (\yb-\xa) -- (\ya-\xb);
   }
 \foreach \xa/\xb in {0/1,1/2,2/3,3/4,4/5,5/6,6/7,7/8}
  \foreach \ya/\yb in {3/2,5/6}
   {
    \draw [thick,loosely dotted] (\ya-\xa) -- (\yb-\xa);
    \draw [thick,loosely dotted] (\yb-\xa) -- (\ya-\xb);
   }
 %draw dashed lines
 \draw [dashed] (0,-.5) -- (0,7.5); 
 \draw [dashed] (16,-.5) -- (16,7.5);
 \draw [decorate,decoration=brace] (5.5,-.5) -- node [above] {$C'$} (12.5,-.5);
 \draw [decorate,decoration={brace,mirror}] (-.5,7.5) -- node [below] {$C'$} (2.5,7.5);
 \draw [decorate,decoration={brace,mirror}] (15.5,7.5) -- node [below] {$C'$} (16.5,7.5);
 \draw [decorate,decoration={brace,mirror}] (5.5,7.5) -- node [below] {$C_{k-1}$} (12.5,7.5);
\end{tikzpicture} } \]
$T_*^{\perp}$ is equivalent to $\mathcal C_{k-1} \times \mathcal C'$ where $\mathcal C'$ is the cluster category of type $D_{k-1}$, $A_3$ or $A_1 \times A_1$. The part of $\mathcal C'$ that is not directly above $T_*$ is equivalent to a triangle of order $n-k-2$. This follows from the well-known structure of the derived category of type $D_n$ (for instance, one can use the ``knitting'' technique to see this). But this means that $T$ can have at most $k-1$ summands in the subcategory equivalent to $\mathcal C_{k-1}$, and at most $n-k-2$ summands that are $\beta$-objects in $\mathcal C'$ since no summands can be directly above $T_*$. Hence there are at most $1+k-1+n-k-2=n-2$ $\beta$-objects in total, and so there must be at least 2 $\alpha$-objects. 
\end{proof}

We now describe how the indecomposable summands of a cluster-tilting object are distributed in the cluster category. We will show how different choices of $\alpha$-objects give rise to the different types of cluster-tilted algebras. The way we group the $\alpha$-objects is as follows:

\begin{itemize}
\item[(a)] Some $\alpha$-object $A$ and $\phi A$,
\item[(b)] two $\alpha$-objects $A$ and $A'$ such that $A' \neq \phi A$ or,
\item[(c)] three or more $\alpha$-objects.
\end{itemize}

In the following lemmas we will assume that $T$ is a cluster-tilting object in $\mathcal C_H$ inducing a cluster-tilted algebra $B=\End(T)^{\op}$ with quiver $Q_B$. We will also assume that the decomposition $T=T_1 \oplus T_2$ is the decomposition from Corollary~\ref{An-delen_bevares}.

\begin{lemma} \label{lemma.classify.1}
If $T_1 = A \oplus \phi A$, then $Q_B$ is of the following type:
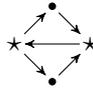
\[ \begin{tikzpicture}[baseline=-5pt,scale=.5,yscale=-1]
 \node (1) at (0,1) [inner sep=1pt] {$\star$};
 \node (2) at (1,0) [vertex] {};
 \node (3) at (1,2) [vertex] {};
 \node (4) at (2,1) [inner sep=1pt] {$\star$};
 \draw [->] (1) -- (2);
 \draw [->] (1) -- (3);
 \draw [->] (2) -- (4);
 \draw [->] (3) -- (4);
 \draw [->] (4) -- (1);
\end{tikzpicture} \]
where the $\star$-vertices are connecting vertices of possibly empty $\mathcal A$-triangles. 
\end{lemma}

\begin{proof}
We know that $A^{\perp}$ is equivalent to $\mathcal C_{n-1}$. We also know that $\phi A$ is an object in the outmost $\tau$-orbit of $A^{\perp}$, and thus can be considered equivalent to the object induced by the projective-injective module for some  embedding of $\mod kQ_{n-1}$ in $A^{\perp}$, which in particular is $\phi A \cup \mathcal B$. Thus $\phi A$ corresponds to a connecting vertex $\bullet_{\phi A}$ in an $\mathcal A$-triangle $Q'$.

If there is exactly one arrow going into or out of $\bullet_{\phi A}$, the vertex this arrow comes from (or goes to) is again a connecting vertex $\star$ for the subquiver $Q' \setminus \bullet_{\phi A}$. It is easy to see from the definition of an $\mathcal A$-triangle that $Q' \setminus \bullet_{\phi A}$ is again an $\mathcal A$-triangle. Furthermore, since $\Hom_{\mathcal C}(X,A)=\Hom_{\mathcal C}(X,\phi A)$ and $\Hom_{\mathcal C}(A,Y)=\Hom_{\mathcal C}(\phi A,Y)$ for all $X,Y \in \mathcal C$, we see that $Q$ is of the desired type but with one empty $\cA$-triangle.

If $\bullet_{\phi A}$ is a vertex with exactly one arrow going into it and one going out of it in some $3$-cycle, it is easy to see that $Q' \setminus \bullet_{\phi A}$ is two $\mathcal A$-triangles connected by the remaining arrow of the previous $3$-cycle. By the same $\Hom$-argument as in the previous paragraph we can add a vertex corresponding to $A$ which is part of a $3$-cycle sharing one arrow with the $3$-cycle of $A'$ (since any map factoring through $A'$ factors through $A$). Thus  $Q$ is of the desired type and without empty $\cA$-triangles.
\end{proof}

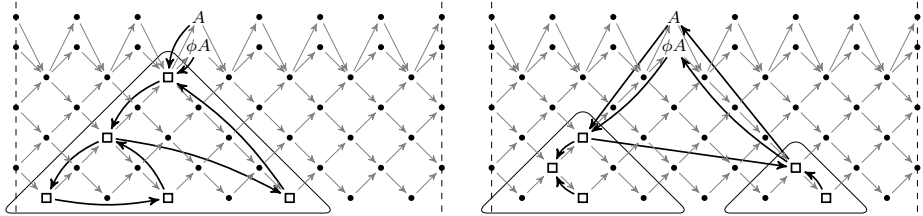
\begin{figure}[htb]
\[ \scalebox{.8} { \begin{tikzpicture}[scale=.5,yscale=-1]
 % distribute vertices
 \foreach \x in {0,...,7}
  \foreach \y in {1,2,4,6}
   \node (\y-\x) at (\x*2,\y) [vertex] {};
 \foreach \x in {0,...,6}
  \foreach \y in {3,5,7}
   \node (\y-\x) at (\x*2+1,\y) [vertex] {};
 \replacevertex{(1-3)}{[tvertex] {$A$}}
 \replacevertex{(2-3)}{[tvertex] {$\phi A$}}
 \replacevertex{(3-2)}{[mvertex] {}}
 \replacevertex{(5-1)}{[mvertex] {}}
 \replacevertex{(7-0)}{[mvertex] {}}
 \replacevertex{(7-2)}{[mvertex] {}}
 \replacevertex{(7-4)}{[mvertex] {}}
 %grey arrows
 \draw [<-,thick,bend right=10] (7-2) to (7-0); 
 \draw [<-,thick,bend right=20] (7-0) to (5-1);
 \draw [<-,thick,bend right=20] (5-1) to (7-2);
 \draw [<-,thick,bend right=20] (5-1) to (3-2);
 \draw [<-,thick,bend left=20] (3-2) to (2-3);
 \draw [<-,thick,bend right=20] (3-2) to (1-3);
 \draw [<-,thick,bend right=10] (3-2) to (7-4);
 \draw [<-,thick,bend left=10] (7-4) to (5-1);
 %triangle
  \draw [rounded corners=5pt] (5,2) -- (10.5,7.5) -- (-.5,7.5) -- cycle;
 %draw arrows
 \foreach \xa/\xb in {0/1,1/2,2/3,3/4,4/5,5/6,6/7}
  \foreach \ya/\yb in {1/3,2/3,4/3,4/5,6/5,6/7}
   {
    \draw [->,color=black!50] (\ya-\xa) -- (\yb-\xa);
    \draw [->,color=black!50] (\yb-\xa) -- (\ya-\xb);
   }
 %draw dashed lines
 \draw [dashed] (0,.5) -- (0,7.5); 
 \draw [dashed] (14,.5) -- (14,7.5);
\end{tikzpicture} \quad
\begin{tikzpicture}[scale=.5,yscale=-1]
 % distribute vertices
 \foreach \x in {0,...,7}
  \foreach \y in {1,2,4,6}
   \node (\y-\x) at (\x*2,\y) [vertex] {};
 \foreach \x in {0,...,6}
  \foreach \y in {3,5,7}
   \node (\y-\x) at (\x*2+1,\y) [vertex] {};
 \replacevertex{(1-3)}{[tvertex] {$A$}}
 \replacevertex{(2-3)}{[tvertex] {$\phi A$}}
 \replacevertex{(5-1)}{[mvertex] {}}
 \replacevertex{(6-5)}{[mvertex] {}}
 \replacevertex{(6-1)}{[mvertex] {}}
 \replacevertex{(7-1)}{[mvertex] {}}
 \replacevertex{(7-5)}{[mvertex] {}}
 %grey arrows
 \draw [<-,thick,bend right=20] (6-1) to (5-1); 
 \draw [<-,thick,bend left=20] (6-1) to (7-1);
 \draw [<-,thick] (5-1) to (1-3);
 \draw [<-,thick,bend left=10] (5-1) to (2-3);
 \draw [<-,thick] (1-3) to (6-5);
 \draw [<-,thick,bend left=10] (2-3) to (6-5);
 \draw [<-,thick] (6-5) to (5-1);
 \draw [<-,thick,bend right=20] (6-5) to (7-5);
 %triangle
  \draw [rounded corners=5pt] (3,4) -- (6.5,7.5) -- (-.5,7.5) -- cycle;
  \draw [rounded corners=5pt] (10,5) -- (12.5,7.5) -- (7.5,7.5) -- cycle;
 %draw arrows
 \foreach \xa/\xb in {0/1,1/2,2/3,3/4,4/5,5/6,6/7}
  \foreach \ya/\yb in {1/3,2/3,4/3,4/5,6/5,6/7}
   {
    \draw [->,color=black!50] (\ya-\xa) -- (\yb-\xa);
    \draw [->,color=black!50] (\yb-\xa) -- (\ya-\xb);
   }
 %draw dashed lines
 \draw [dashed] (0,.5) -- (0,7.5); 
 \draw [dashed] (14,.5) -- (14,7.5);
\end{tikzpicture} } \]
\caption{Diagrams of cluster-tilting objects as in Lemma~\ref{lemma.classify.1}.}
\end{figure}

\begin{lemma} \label{lemma.classify.23}
If $T_1 = A \oplus A'$ such that $A'$ is different from $\phi A$, then $Q_B$ is of the following type:
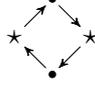
\[ \begin{tikzpicture}[scale=.5,yscale=-1]
 \node (P1) at (2,6) [inner sep=1pt] {$\star$};
 \node (P2) at (3,5) [vertex] {};
 \node (P3) at (3,7) [vertex] {};
 \node (P4) at (4,6) [inner sep=1pt] {$\star$};
 \draw [->] (P1) -- (P2);
 \draw [->] (P3) -- (P1);
 \draw [->] (P2) -- (P4);
 \draw [->] (P4) -- (P3);
\end{tikzpicture} \]
where the $\star$-vertices are connecting vertices of possibly empty $\mathcal A$-triangles.
\end{lemma}

\begin{proof}
We are in the setup of Lemma~\ref{triangler2}. Note that if $A'=\tau^{\pm} \phi A$ then one of the $\Delta$'s is empty. The claim now follows from Corollary~\ref{triangelkorrespondanse}.
\end{proof}

\begin{figure}[htb]
\[ \scalebox{.8}{ \begin{tikzpicture}[scale=.5,yscale=-1]
 % distribute vertices
 \foreach \x in {0,...,7}
  \foreach \y in {1,2,4,6}
   \node (\y-\x) at (\x*2,\y) [vertex] {};
 \foreach \x in {0,...,6}
  \foreach \y in {3,5,7}
   \node (\y-\x) at (\x*2+1,\y) [vertex] {};
 \replacevertex{(1-3)}{[tvertex] {$A$}}
 \replacevertex{(2-2)}{[tvertex] {$A'$}}
 \replacevertex{(3-2)}{[mvertex] {}}
 \replacevertex{(5-1)}{[mvertex] {}}
 \replacevertex{(6-1)}{[mvertex] {}}
 \replacevertex{(7-1)}{[mvertex] {}}
 \replacevertex{(7-4)}{[mvertex] {}}
 %grey arrows
 \draw [<-,thick,bend right=20] (6-1) to (5-1); 
 \draw [<-,thick,bend left=20] (6-1) to (7-1);
 \draw [<-,thick,bend right=20] (5-1) to (3-2);
 \draw [<-,thick,bend right=10] (3-2) to (7-4);
 \draw [<-,thick,bend left=10] (7-4) to (5-1);
 \draw [<-,thick,bend left=20] (2-2) to (3-2);
 \draw [<-,thick,bend right=20] (3-2) to (1-3);
 %triangle
  \draw [rounded corners=5pt] (5,2) -- (10.5,7.5) -- (-.5,7.5) -- cycle;
 %draw arrows
 \foreach \xa/\xb in {0/1,1/2,2/3,3/4,4/5,5/6,6/7}
  \foreach \ya/\yb in {1/3,2/3,4/3,4/5,6/5,6/7}
   {
    \draw [->,color=black!50] (\ya-\xa) -- (\yb-\xa);
    \draw [->,color=black!50] (\yb-\xa) -- (\ya-\xb);
   }
 %draw dashed lines
 \draw [dashed] (0,.5) -- (0,7.5); 
 \draw [dashed] (14,.5) -- (14,7.5);
\end{tikzpicture} \quad
\begin{tikzpicture}[scale=.5,yscale=-1]
 % distribute vertices
 \foreach \x in {0,...,7}
  \foreach \y in {1,2,4,6}
   \node (\y-\x) at (\x*2,\y) [vertex] {};
 \foreach \x in {0,...,6}
  \foreach \y in {3,5,7}
   \node (\y-\x) at (\x*2+1,\y) [vertex] {};
 \replacevertex{(1-3)}{[tvertex] {$A$}}
 \replacevertex{(2-6)}{[tvertex] {$A'$}}
 \replacevertex{(6-1)}{[mvertex] {}}
 \replacevertex{(5-4)}{[mvertex] {}}
 \replacevertex{(7-0)}{[mvertex] {}}
 \replacevertex{(7-4)}{[mvertex] {}}
 \replacevertex{(6-5)}{[mvertex] {}}
 %grey arrows
 \draw [<-,thick,bend right=20] (7-0) to (6-1); 
 \draw [<-,thick] (6-1) to (1-3);
 \draw [<-,thick] (1-3) to (5-4);
 \draw [<-,thick,bend right=20] (5-4) to (6-5);
 \draw [<-,thick,bend left=20] (7-4) to (6-5);
 \draw [<-,thick,bend right=10] (5-4) to (2-6);
 \draw [<-,thick,bend left=90] (2-6) to (6-1);
 %triangle
  \draw [rounded corners=5pt] (2,5) -- (4.5,7.5) -- (-.5,7.5) -- cycle;
  \draw [rounded corners=5pt] (9,4) -- (12.5,7.5) -- (5.5,7.5) -- cycle;
 %draw arrows
 \foreach \xa/\xb in {0/1,1/2,2/3,3/4,4/5,5/6,6/7}
  \foreach \ya/\yb in {1/3,2/3,4/3,4/5,6/5,6/7}
   {
    \draw [->,color=black!50] (\ya-\xa) -- (\yb-\xa);
    \draw [->,color=black!50] (\yb-\xa) -- (\ya-\xb);
   }
 %draw dashed lines
 \draw [dashed] (0,.5) -- (0,7.5); 
 \draw [dashed] (14,.5) -- (14,7.5);
\end{tikzpicture} } \]
\caption{Diagrams of cluster-tilting objects as in Lemma~\ref{lemma.classify.23}. In the left example, one triangle is empty, and hence disappears.}
\end{figure}
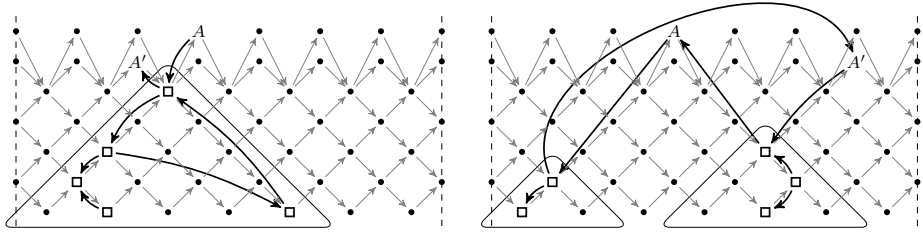

\begin{lemma}  \label{lemma.classify.4}
If $T_1 = A_1 \oplus A_2 \oplus \ldots \oplus A_n$ where $n \geq 3$, then $Q_B$ is of the following type:
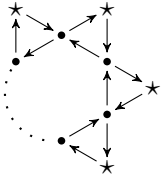
\[ \begin{tikzpicture}[scale=.5,yscale=-1]
 \node (P1) at (2,9.4) [vertex] {};
 \node (P2) at (3.2,8.7) [vertex] {};
 \node (P3) at (4.4,9.4) [vertex] {};
 \node (P4) at (4.4,10.8) [vertex] {};
 \node (P5) at (3.2,11.5) [vertex] {};
 \node (P6) at (2,8) [inner sep=1pt] {$\star$};
 \node (P7) at (4.4,8) [inner sep=1pt] {$\star$};
 \node (P8) at (5.6,10.1) [inner sep=1pt] {$\star$};
 \node (P9) at (4.4,12.2) [inner sep=1pt] {$\star$};
 \draw [->] (P2) -- (P1);
 \draw [->] (P3) -- (P2);
 \draw [->] (P4) -- (P3);
 \draw [->] (P5) -- (P4);
 \draw [->] (P1) -- (P6);
 \draw [->] (P6) -- (P2);
 \draw [->] (P2) -- (P7);
 \draw [->] (P7) -- (P3);
 \draw [->] (P3) -- (P8);
 \draw [->] (P8) -- (P4);
 \draw [->] (P4) -- (P9);
 \draw [->] (P9) -- (P5);
 \draw [thick, loosely dotted] (P1) .. controls (1.4,10.45) and (2,11.5) .. (P5);
\end{tikzpicture} \]
where the $\star$-vertices are connecting vertices of possibly empty $\mathcal A$-triangles.

\end{lemma}

\begin{proof}
Assume that $T_1 = A_1 \oplus A_2 \oplus \ldots \oplus A_n$ where $n \geq 3$, and that $A_t = \tau^{-s_t}A_1$ or $\tau^{-s_t}(\phi A_1)$ for $2 \leq t\leq n$. Then $\bigcap_{1\leq s\leq t}\mathcal B(A_s)$ is a union of $n$ disjoint, possibly empty, triangles in $\mathcal B(A_1)$ by Lemma~\ref{triangler2}. If $A_i = \tau^{-1}(\phi A_{i+1})$ for some $i$, the corresponding triangle is empty. In particular there will be an $n$-cycle 
\[A_1 \rightarrow A_2 \rightarrow A_3 \rightarrow \cdots \rightarrow A_n \rightarrow A_1 \]
where the arrow $A_r \rightarrow A_{r+1}$ has a triangle of order $k-1$ "attached" if $A_{r+1} = \tau^{-k}A_r$:
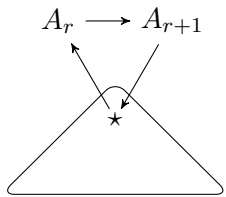
\[ \begin{tikzpicture}[yscale=-1]
 \node (S) at (0,0) [inner sep=1pt] {$\star$};
 \node (A1) at (240:1.5) {$A_r$};
 \node (A2) at (300:1.5) {$A_{r+1}$};
 \draw [->] (S) -- (A1);
 \draw [->] (A1) -- (A2);
 \draw [->] (A2) -- (S);
 \draw [rounded corners=5pt] (0,-.5) -- (1.5,1) -- (-1.5,1) -- cycle;
\end{tikzpicture} \]
The relations also follow from Lemma~\ref{triangler2}. Now we are done by Corollary~\ref{triangelkorrespondanse}. 
\end{proof}

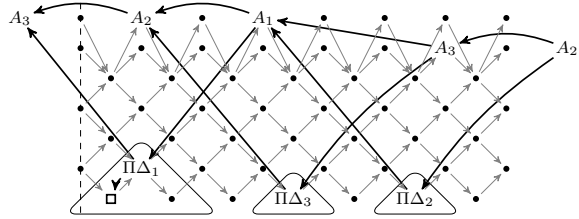
\begin{figure}[htb]
\[ \scalebox{.8}{ \begin{tikzpicture}[scale=.5,yscale=-1]
 % distribute vertices
 \foreach \x in {0,...,7}
  \foreach \y in {1,2,4,6}
   \node (\y-\x) at (\x*2,\y) [vertex] {};
 \foreach \x in {0,...,6}
  \foreach \y in {3,5,7}
   \node (\y-\x) at (\x*2+1,\y) [vertex] {};
 \replacevertex{(1-3)}{[tvertex] {$A_1$}}
 \replacevertex{(1-1)}{[tvertex] {$A_2$}}
 \replacevertex{(2-6)}{[tvertex] {$A_3$}}
 \replacevertex{(6-1)}{[tvertex] {$\Pi\Delta_1$}}
 \replacevertex{(7-3)}{[tvertex] {$\Pi\Delta_3$}}
 \replacevertex{(7-5)}{[tvertex] {$\Pi\Delta_2$}}
 \replacevertex{(7-0)}{[mvertex] {}}
 \node (A1) at (16,2) [tvertex] {$A_2$};
 \node (A2) at (-2,1) [tvertex] {$A_3$};
 %grey arrows
 \draw [<-,thick,bend right=20] (7-0) to (6-1); 
 \draw [<-,thick] (6-1) to (1-3);
 \draw [<-,thick,bend right=10] (7-3) to (2-6);
 \draw [<-,thick,bend right=10] (7-5) to (A1);
 \draw [<-,thick] (A2) to (6-1);
 \draw [<-,thick] (1-3) to (7-5);
 \draw [<-,thick] (1-1) to (7-3);
 \draw [<-,thick,bend right=20] (A2) to (1-1);
 \draw [<-,thick,bend right=20] (1-1) to (1-3);
 \draw [<-,thick] (1-3) to (2-6);
 \draw [<-,thick,bend right=20] (2-6) to (A1);
 %triangle
  \draw [rounded corners=5pt] (2,5) -- (4.5,7.5) -- (-.5,7.5) -- cycle;
  \draw [rounded corners=5pt] (7,6) -- (8.5,7.5) -- (5.5,7.5) -- cycle;
  \draw [rounded corners=5pt] (11,6) -- (12.5,7.5) -- (9.5,7.5) -- cycle;
 %draw arrows
 \foreach \xa/\xb in {0/1,1/2,2/3,3/4,4/5,5/6,6/7}
  \foreach \ya/\yb in {1/3,2/3,4/3,4/5,6/5,6/7}
   {
    \draw [->,color=black!50] (\ya-\xa) -- (\yb-\xa);
    \draw [->,color=black!50] (\yb-\xa) -- (\ya-\xb);
   }
 %draw dashed lines
 \draw [dashed] (0,.5) -- (0,7.5); 
\end{tikzpicture} } \]
\caption{Diagram of a cluster-tilting object as in Lemma~\ref{lemma.classify.4}.}
\end{figure}

This finishes the proof of Theorem \ref{main_theorem}. Also note that we obtained a new proof of the main result of \cite{Dagfinn}.

\begin{remark}
The cluster category of type $D_{n}$ has also been modeled geometrically as a polygon with $n$ vertices and a puncture in its center in \cite{Schiffler}. Here the cluster-tilting objects correspond to triangulations of the polygon. Some results in this subsection could be proved using this approach. One can, for instance, derive Lemma~\ref{lemma.alpha-objects} from the fact that the $\alpha$-objects correspond to arcs in the punctured polygon that are incident to the puncture. Every triangulation contains at least two arcs incident to the puncture and thus every cluster tilting object has at least two $\alpha$-objects as summands\footnote{We want to thank Ralf Schiffler for this remark.}. 
\end{remark}

\end{document}